\documentclass[11pt,reqno,a4paper]{amsart}
\usepackage{amssymb,amscd,amsmath}
\usepackage{pspicture}
\usepackage{graphicx}

\usepackage{mathptmx}

\usepackage[matrix,arrow,curve,frame]{xy}    % XY-pic diagram pac

\xymatrixcolsep{1.9pc}                          % Adjust size of diagrams.
\xymatrixrowsep{1.9pc}
\newdir{ >}{{}*!/-5pt/\dir{>}}                  % Make better tailed arrows

 \calclayout

\makeatletter

\renewcommand{\subsection}{\@startsection{subsection}{1}{0pt}{-3.25ex plus -1ex minus-.2ex}{1.5ex plus.2ex}{\normalfont\it}}
\renewcommand{\section}{\@startsection{section}{1}{\parindent}{3.5ex plus 1ex minus .2ex}{2.3ex plus.2ex}{\sc}}

\renewcommand{\phi}{\varphi}
\renewcommand{\leq}{\leqslant}
\renewcommand{\geq}{\geqslant}
\renewcommand{\epsilon}{\varepsilon}

\renewcommand{\kappa}{\varkappa}

\DeclareMathOperator{\Ext}{Ext}

\DeclareMathOperator{\Hom}{Hom} 
 
 \DeclareMathOperator{\id}{id}

\DeclareMathOperator{\chr}{char} 
 
\DeclareMathOperator{\coker}{Coker} \DeclareMathOperator{\nis}{nis}

\newcommand{\cc}{\mathcal}
\newcommand{\bb}{\mathbb}

\newcommand{\op}{{\textrm{\rm op}}}

\newtheorem{thm}{Theorem}[section]
\newtheorem{prop}[thm]{Proposition}

\newtheorem{cor}[thm]{Corollary}
\newtheorem{lem}[thm]{Lemma}

\newtheorem{rem}[thm]{Remark}

\newtheorem{defs}[thm]{Definition}
\newtheorem{notn}[thm]{Notation}
\newtheorem{constr}[thm]{Construction}
\newtheorem{claim}[thm]{Claim}

\makeatother

\begin{document}

\footskip30pt

%\baselineskip=1.5\baselineskip

\title{Surjectivity of the \'{e}tale excision map for homotopy invariant framed presheaves}
%\author{Grigory Garkusha}
%\address{Department of Mathematics, Swansea University, Singleton Park, Swansea SA2 8PP, United Kingdom}
%\email{g.garkusha@swansea.ac.uk}
\author{Andrei Druzhinin} %\author{Ivan Panin}
\address{Chebyshev Laboratory, 14-line of Vasilievsky Irland, 29b %линия Васильевского острова, 29б,
St. Petersburg, Russia}
%\address{St. Petersburg Department of V.A.Steklov Institute of Mathematics of the Russian Academy of Sciences, 191023, 27 Fontanka, St. Petersburg, Russia}

\author{Ivan Panin}
\address{St. Petersburg Branch of V. A. Steklov Mathematical Institute,
Fontanka 27, 191023, St. Petersburg, Russia}

%\email{paniniv@gmail.com}

%\address{Institute for Advanced Study, Einstein Drive, Princeton, NJ, 08540, USA}

%\email{panin@math.ias.edu}

\thanks{The authors gratefully acknowledge excellent working conditions and support
provided by the RCN Frontier Research Group Project no. 250399 “Motivic Hopf equations" at University of Oslo.}

\thanks{The second author gratefully acknowledge support of the
RFBR grant 16-01-00750-a.}

%\thanks{The second author thanks the
%Institute for Advanced Study for the support and for the kind
%hospitality during his visit in the Second Term of 2014--2015.}

\begin{abstract}
The category of framed correspondences $Fr_*(k)$, framed presheaves
and framed shea\-ves were invented by Voevodsky in his unpublished
notes~\cite{Voe2}. Based on the notes~\cite{Voe2} a
new approach to the classical Morel--Voevodsky motivic
stable homotopy theory was developed by \\
G.Garkusha and I.Panin in \cite{GP3}.

The purpose of this paper is
to prove Theorem \ref{Few_On_P_w_Tr_Introduction}
stating that if the ground field $k$ is infinite, then the surjectivity of the
\'{e}tale excision property is true for any $\bb A^1$-invariant stable radditive
framed presheaf of Abelian groups $\cc F$. The injectivity of the \'{e}tale excision was proved in
\cite{GP4}. The surjectivity of the \'{e}tale excision was proved in \cite{GP4} if the ground field
is infinite of characteristic not 2. In this preprint the surjectivity of the \'{e}tale excision
is proved in the case of any infinite ground field.

%The significance of this result is this: 
{\it As explained in \cite[Intoduction]{GP3} all the results of \cite{GP4}, \cite{AGP}, \cite{GNP} and \cite{GP3} are true now
automatically without any restrictions on the characteristic of the ground field. 
%Their proofs run 
%literally in the same way as in the case of characteristic not 2.
}

%{\it Thus, the \'{e}tale excision property is true for any $\bb A^1$-invariant stable radditive
%framed presheaf of Abelian groups in the case of any infinite ground field.}

%%%the associated Nisnevich
%%%sheaf $\cc F_{\nis}$ is strictly $\bb A^1$-invariant and $\sigma$-stable whenever the base field
%%%$k$ is infinite perfect of characteristic different from 2.

%Point out that as
%%%The original Voevodsky theorem~\cite{Voe1} and all other known theorems of the same fashion
%as well as all other known results of the same shape
%as Theorem \ref{thm1}
%%%{\it are not suitable} for the big frame motives theory. Indeed,
%%%bigraded presheaves of $\bb A^1$-homotopy groups of a bispectra $E\in SH(k)$
%%%are naturally $Fr_*(k)$-presheaves (= framed presheaves), however they
%%%are in {\it a no reasonable way} presheaves
%%%with transfers as in the sense of \cite{Voe1}, so in the sense of other indicated works.
%Finally point out
%that the paper is inspired by Voevodsky's
%one~\cite{Voe1}, rather than used the machinery of standard triples.
%%%Inspite of
%The standard triple machinery of Voevodsky~\cite{Voe1} {\it does not work } in our case,
%however the paper \ is definitely inspired by the Voevodsky's one
%~\cite{Voe1}.
%since
%it works with divisors of rational functions.
%But we are forced to work systematically with divisors of regular functions.

%This result and the paper are inspired by Voevodsky's
%paper~\cite{Voe1}.
\end{abstract}



\maketitle

\thispagestyle{empty} \pagestyle{plain}

\newdir{ >}{{}*!/-6pt/@{>}} %this command is to define the arrow of the type \ar@{ >->} (spacing of arrows is needed sometimes)

\tableofcontents

\section{Introduction}
The category of framed correspondences $Fr_*(k)$, framed presheaves
and framed shea\-ves were invented by Voevodsky in his unpublished
notes~\cite{Voe2}. Based on the notes~\cite{Voe2} a
new approach to the classical Morel--Voevodsky motivic
stable homotopy theory was developed in \cite{GP3}.

%{\it The author is very greatful to Ivan Panin
%for very careful reading of the preliminary version of the manuscript.
%As a result the text was significantly improved.
%}
The purpose of this paper is
to prove Theorem \ref{Few_On_P_w_Tr}
stating that if the ground field $k$ is infinite, then the surjectivity of the
\'{e}tale excision property is true for any $\bb A^1$-invariant stable radditive
framed presheaf of Abelian groups $\cc F$. The injectivity of the \'{e}tale excision was proved in
\cite{GP4}. The surjectivity of the \'{e}tale excision was proved in \cite{GP4} if the ground field
is infinite of characteristic not 2. In this preprint the surjectivity, Theorem \ref{Few_On_P_w_Tr_Introduction},
of the \'{e}tale excision
is proved in the case of any infinite ground field.
That is the main aim of the paper.
%Theorem \ref{Few_On_P_w_Tr_Introduction} stated below is the {\it core\/} of
%the theory of big framed motives of~\cite{GP3}.
%The main goal of this paper is
%to prove Theorem \ref{Few_On_P_w_Tr_Introduction}.
More detailed statement is given in
Theorem \ref{Few_On_P_w_Tr}.

In the rest of the Introduction we state Theorem \ref{Few_On_P_w_Tr_Introduction}.
By Definition \ref{stab} each element $a\in Fr_n(X,Y)$ has its
support $Z_a$. It is a closed subset in $X\times \bb A^n$ which is
finite over $X$ and determined by $a$ uniquely. If the support $Z_a$
of an element $a\in Fr_n(X,Y)$ is a disjoint union of $Z_1$ and
$Z_2$, then the element $a$ determines uniquely two elements $a_1$
and $a_2$ in $Fr_n(X,Y)$ such that the support of $a_i$ is $Z_i$.
This is explained in Definition~\ref{stab} below. Therefore one can
form a subgroup $A(X,Y)$ of the free abelian group $\bb ZFr_n(X,Y)$
generated by elements of the form $1\cdot a-1\cdot a_1-1\cdot a_2$,
where $a\in Fr_n(X,Y)$ runs over those elements whose support $Z_a$
is a disjoint union of $Z_1$ and $Z_2$, and $a_1$, $a_2$ are the
elements determined by $a$ as just above.

The main result, Theorem~\ref{Few_On_P_w_Tr}, can be reformulated in terms of
$\bb ZF_*$-presheaves of abelian groups on smooth algebraic
varieties $Sm/k$. Recall that $\bb ZF_*(k)$ is defined
in~\cite[Definition~8.3]{GP3} as an additive category whose objects
are those of $Sm/k$ and Hom-groups are defined as follows (see
Definition \ref{stab}). We set for every $n\geq 0$ and $X,Y\in
Sm/k$,
   $$\bb ZF_n(X,Y)=\bb ZFr_n(X,Y)/A(X,Y).$$
%%%   $$\bb ZF_n(X,Y)=\bb ZFr_n(X,Y)/\langle Z_1\sqcup Z_2-Z_1-Z_2\rangle,$$
%%%where $Z_1,Z_2$ are supports of correspondences.
In other words,
$\bb ZF_n(X,Y)$ is a free abelian group generated by the framed
correspondences of level $n$ with connected supports. We then set
   $$\Hom_{\bb ZF_*(k)}(X,Y):=\bigoplus_{n\geq 0}\bb ZF_n(X,Y).$$
The canonical morphisms $Fr_*(X,Y) \to \Hom_{\bb ZF_*(k)}(X,Y)$
define a functor $R: Fr_*(k) \to \bb ZF_*(k)$, which is the identity
on objects. For any $\bb A^1$-invariant quasi-stable $\bb
ZF_*$-presheaf of Abelian groups $\cc F$ the functor $\cc F \circ R:
Fr_*(k)^{\op} \to Ab$ is $\bb A^1$-invariant quasi-stable radditive
framed presheaf of Abelian groups.

By definition, a $Fr_*$-presheaf $\cc F$ of Abelian groups is
stable if for any $k$-smooth variety the pull-back map
$\sigma^*_X: \cc F(X) \to \cc F(X)$
equals the identity map, where
$\sigma_X=(X\times 0, X\times \bb A^1, t; pr_X) \in Fr_1(X,X)$.
In turn, $\cc F$ is quasi-stable if for any $k$-smooth variety the pull-back map
$\sigma^*_X: \cc F(X) \to \cc F(X)$
is an isomorphism. Also, recall that $\cc F$ is  radditive if
$\cc F(\emptyset)=\{0\}$ and $\cc F(X_1\sqcup X_2)=\cc F(X_1)\times \cc F(X_2)$.

For any $\bb A^1$-invariant stable (respectively quasi-stable)
radditive $Fr_*$-presheaf of Abelian groups $G$ there is a unique
$\bb A^1$-invariant stable (respectively quasi-stable) $\bb
ZF_*$-presheaf of Abelian groups $\cc F$ such that $G=\cc F \circ
R$. This follows easily from the Additivity Theorem of~\cite{GP3}.

Therefore {\it the category of $\bb A^1$-invariant stable
(respectively quasi-stable) radditive framed pre\-sheaves of Abelian
groups is equivalent to the category of $\bb A^1$-invariant stable
(respectively quasi-stable) $\bb ZF_*$-pre\-sheaves of Abelian
groups}.

The latter means that the main result formulated in the abstract is
equivalent to the following
\begin{thm}
\label{Few_On_P_w_Tr_Introduction}
For any $\bb A^1$-invariant quasi-stable $\bb ZF_*$-presheaf of abelian groups $\mathcal F$
the surjective part of the \'{e}tale excision is true.
\end{thm}
For more detailed statement see Theorem \ref{Few_On_P_w_Tr}.

%\begin{thm}[Main]\label{thm1}
%For any $\bb A^1$-invariant quasi-stable $\bb ZF_*$-presheaf of
%abelian groups $\cc F$, the associated Nisnevich sheaf $\cc
%F_{\nis}$ is $\bb A^1$-invariant whenever the base field $k$ is
%infinite of characteristic different from 2. Moreover, if the base
%field $k$ is infinite perfect of characteristic different from 2,
%then every $\bb A^1$-invariant quasi-stable Nisnevich framed sheaf
%of Abelian groups is strictly $\bb A^1$-invariant and quasi-stable.
%Furthermore, the same statements are true in characteristic 2 if we
%also assume that the $\bb ZF_*$-presheaf of abelian groups $\cc F$
%is a presheaf of $\bb Z[1/2]$-modules.
%\end{thm}

%\subsubsection*{Acknowledgements}
%The author is very greatful to Ivan Panin
%for very careful reading of the preliminary version of the manuscript.
%As a result the text was significantly improved.
The authors would like to thank
Alexey Ananyevskiy and Grigory Garkusha for their deep interest to the topic
of the present preprint. 

In the rest of the paper we suppose that {\it the base field $k$ is infinite}.

%The authors gratefully acknowledge excellent working conditions and support 
%provided by the RCN Frontier Research Group Project no. 250399 “Motivic Hopf equations" at University of Oslo. 
%Andrey Druzhinin and Alexander Neshitov for many
%helpful discussions.

\section{Recollections on Voevodsky's framed correspondences}\label{ohoho}
We recollect here basic facts for framed correspondences
in the sense of Voevodsky~\cite{Voe2} and for linear framed correspondences
in the sense of \cite[Section 8]{GP3}.
This Section is  mainly copied from \cite[Section 2]{GP3} and \cite[Section 8]{GP3}.
We include it here
for convenience of the reader.
We start
with preparations.

Let $S$ be a scheme and $Z$ be a closed subscheme. Recall that an
{\it \'{e}tale neighborhood of $Z$ in $S$\/} is a triple
$(W',\pi':W'\to S,s': Z\to W')$ satisfying the conditions:

\indent (i) $\pi'$ is an \'{e}tale morphism;

\indent (ii) $\pi'\circ s'$ coincides with the inclusion
$Z\hookrightarrow S$ (thus $s'$ is a closed embedding);

\indent (iii) $(\pi')^{-1}(Z)=s'(Z)$

A morphism between two \'{e}tale neighborhoods
$(W',\pi',s')\to(W'',\pi'',s'')$ of $Z$ in $S$
%in this category
is a morphism $\rho:W'\to W''$ such that $\pi''\circ\rho=\pi'$ and
$\rho\circ s'=s''$. Note that such $\rho$ is automatically \'etale.
%by~\cite[VI.4.7]{LNM146}.

\begin{defs}[Voevodsky~\cite{Voe2}]\label{frame_corr}{\rm
For $k$-smooth schemes $X, Y$ and $n\geq 0$ an {\it explicit framed
correspondence  $\Phi$ of level $n$\/} consists of the following
data:

\begin{enumerate}
\item a closed subset $Z$ in $\mathbb A^n_X$ which is finite over $X$;
\item an etale neighborhood $p:U\to\mathbb A^n_X$ of $Z$ in $\mathbb A^n_X$;
\item a collection of regular functions $\phi=(\phi_1,\ldots,\phi_n)$ on $U$
such that $\cap_{i=1}^n \{\phi_i=0\}=Z$;
\item a morphism $g:U\to Y$.
\end{enumerate}
The subset $Z$ will be referred to as the {\it support\/} of the
correspondence. We shall also write triples $\Phi=(U,\phi,g)$ or
quadruples $\Phi=(Z,U,\phi,g)$ to denote explicit framed
correspondences.

Two explicit framed correspondences $\Phi$ and $\Phi'$ of level $n$
are said to be {\it equivalent\/} if they have the same support and
there exists an open neighborhood $V$ of $Z$ in $U \times_{\bb
A^n_X} U'$ such that on $V$, the morphism $g\circ pr$ agrees with
$g' \circ pr'$ and $\phi \circ pr$ agrees with $\phi' \circ pr'$. A
{\it framed correspondence of level $n$} is an equivalence class of
explicit framed correspondences of level $n$.
}\end{defs}

We let $Fr_n(X,Y)$ denote the set of framed correspondences from
$X$ to $Y$. We consider it as a pointed set with the basepoint
being the class $0_n$ of the explicit correspondence with
$U=\emptyset$.

As an example, the sets $Fr_0(X,Y)$ coincide
with the set of pointed morphisms $X_+\to Y_+$. In particular, for a
connected scheme $X$ one has\label{fr0}
   $$Fr_0(X,Y)=\text{Hom}_{Sm/k}(X,Y)\sqcup\{0_0\}.$$

If $f:X'\to X$ is a morphism of schemes and $\Phi=(U,\phi,g)$ an
explicit correspondence from $X$ to $Y$ then
   $$f^*(\Phi):=(U'=U\times_X X',\phi\circ pr,g\circ pr)$$
is an explicit correspondence from $X'$ to $Y$.

\begin{rem}\label{short_Notation}
{\rm
Let $\Phi=(Z,\mathbb A^n_X \xleftarrow{p} U,\phi: U \to \mathbb
A^n_k,g: U\to Y) \in Fr_n(X,Y)$ be an {\it explicit framed
correspondence  of level n}. It can more precisely be written in the
form
\[
((\alpha_1,\alpha_2,\dots,\alpha_n),f,Z,U,(\varphi_1,\varphi_2,\dots,\varphi_n),g)\in Fr_n(X,Y)
\]
where
\begin{itemize}
\item[$\diamond$]
$Z\subset \bb A^n_X$ is a closed subset finite over $X$,
\item[$\diamond$]
an etale neighborhood
$(\alpha_1,\alpha_2,\dots,\alpha_n),f)=p: U \to \mathbb A^n_k \times X$ of $Z$,
\item[$\diamond$]
a collection of regular functions $\phi=(\phi_1,\ldots,\phi_n)$ on $U$
such that $\cap_{i=1}^n \{\phi_i=0\}=Z$;
\item[$\diamond$] a morphism $g:U\to Y$.
\end{itemize}
We shall usually drop $((\alpha_1,\alpha_2,\dots,\alpha_n),f)$ from
notation and just write
\[
(Z,U,(\varphi_1,\varphi_2,\dots,\varphi_n),g)=((\alpha_1,\alpha_2,\dots,\alpha_n),f,Z,U,(\varphi_1,\varphi_2,\dots,\varphi_n),g).
\]

}\end{rem}

The following definition is to describe compositions of framed
correspondences.

\begin{defs}\label{cat_Fr_+}{\rm
Let $X,Y$ and $S$ be $k$-smooth schemes and let
\begin{gather*}
a=((\alpha_1,\alpha_2,\dots,\alpha_n),f,Z,U,(\varphi_1,\varphi_2,\dots,\varphi_n),g)
\end{gather*}
be an explicit correspondence of level $n$ from $X$
to $Y$ and let
\begin{gather*}
b=((\beta_1,\beta_2,\dots,\beta_m),f',Z',U',(\psi_1,\psi_2,\dots,\psi_m),g')\in Fr_m(Y,S)
\end{gather*}
be an explicit correspondence of level $m$ from $Y$ to $S$. We
define their composition as an explicit correspondence of level
$n+m$ from $X$ to $S$ by
\[
((\alpha_1,\alpha_2,\dots,\alpha_n,\beta_1,\beta_2,\dots,\beta_m),f,Z\times_Y Z',U\times_Y U',(\varphi_1,\varphi_2,\dots,\varphi_n,\psi_1,\psi_2,\dots,\psi_m),g').
\]
Clearly, the composition of explicit correspondences respects the
equivalence relation on them and defines associative maps
   \begin{equation*}\label{compos}
    Fr_n(X,Y)\times Fr_m(Y,S)\to Fr_{n+m}(X,S).
   \end{equation*}
}\end{defs}

Given $X, Y\in Sm/k$, denote by $Fr_+(X,Y)$ the set $\bigvee_n
Fr_n(X,Y)$. The composition of framed correspondences defined above
gives a category $Fr_+(k)$. Its objects are those of $Sm/k$ and the
morphisms are given by the sets $Fr_+(X,Y)$, $X, Y\in Sm/k$. Since
the naive morphisms of schemes can be identified with certain framed
correspondences of level zero, we get a canonical functor
   $$Sm/k\to Fr_+(k).$$
{\it The category $Fr_+(k)$ has the zero object.
It is the empty scheme.
}
One can easily see that for a framed correspondence $\Phi:X\to Y$
and a morphism $f:X'\to X$, one has $f^*(\Phi)=\Phi\circ f$.

%\begin{defs}\label{d:boxpairing}{\rm
%Let $X,Y,S$ and $T$ be smooth schemes. There is an \textit{external
%product}
%\[
%Fr_n(X,Y)\times Fr_m(S,T) \xrightarrow{-\boxtimes -}
%Fr_{n+m}(X\times S, Y\times T)
%\]
%given by
%\begin{gather*}
%((\alpha_1,\alpha_2,\dots,\alpha_n),f,Z,U,(\varphi_1,\varphi_2,\dots,\varphi_n),g)\boxtimes
%((\beta_1,\beta_2,\dots,\beta_m),f',Z',U',(\psi_1,\psi_2,\dots,\psi_m),g')=
%\\
%((\alpha_1,\alpha_2,\dots,\alpha_n,\beta_1,\beta_2,\dots,\beta_m),f\times
%f',Z\times Z',U\times
%U',(\varphi_1,\varphi_2,\dots,\varphi_n,\psi_1,\psi_2,\dots,\psi_m),g\times
%g').
%\end{gather*}

%For the constant morphism $c\colon \bb A^1\to pt$, we set (following
%Voevodsky~\cite{Voe2})
%\[
%\Sigma=-\boxtimes (t,c,\{0\},\bb A^1,t,c)\colon Fr_n(X,Y)\to Fr_{n+1}(X,Y)
%\]
%and refer to it as the \textit{suspension}. Note that the map $\Sigma$ coincides with
%the map $a\mapsto \sigma_Y \circ a: Fr_n(X,Y)\to Fr_{n+1}(X,Y)$, where
%$\sigma_Y=(Y\times 0, \bb A^1_Y, t, pr_Y)\in Fr_1(Y,Y)$.
%Also, following Voevodsky~\cite{Voe2}, one puts
%\[
%Fr(X,Y)=\colim(Fr_0(X,Y)\xrightarrow{\Sigma}Fr_1(X,Y)\xrightarrow{\Sigma}\dots
%\xrightarrow{\Sigma} Fr_n(X,Y)\xrightarrow{\Sigma}\dots)
%\]
%and refer to it as the \textit{set stable framed correspondences}.
%The above external product induces external products
%\begin{gather*}
%Fr_n(X,Y)\times Fr(S,T) \xrightarrow{-\boxtimes -} Fr(X\times S, Y\times T),\\
%Fr(X,Y)\times Fr_0(S,T) \xrightarrow{-\boxtimes -} Fr(X\times S, Y\times T).
%\end{gather*}
%}\end{defs}

Recall the definition of the category of linear framed correspondences $\bb ZF_*(k)$ introduced in
\cite[Definition. 8.3]{GP3}.

\begin{defs}\label{stab}{\rm
Let $X$ and $Y$ be smooth schemes. Denote by
\begin{itemize}
\item[$\diamond$]
${\bb Z}Fr_n(X,Y):=\widetilde{\mathbb{Z}}[Fr_n(X,Y)]=\mathbb{Z}[Fr_n(X,Y)]/\mathbb{Z}\cdot
0_n$, i.e the free abelian group generated by the set $Fr_n(X,Y)$
modulo $\mathbb{Z}\cdot 0_n$;

\item[$\diamond$]
${\bb Z}F_n(X,Y):={\bb Z}Fr_n(X,Y)/A$, where $A$ is a subgroup
generated by the elementts
\begin{multline*}
(Z\sqcup Z', U,(\varphi_1,\varphi_2,\dots,\varphi_n),g) - \\
-(Z, U\setminus
Z',(\varphi_1,\varphi_2,\dots,\varphi_n)|_{U\setminus
Z'},g|_{U\setminus Z'}) - (Z',{U\setminus
Z},(\varphi_1,\varphi_2,\dots,\varphi_n)|_{U\setminus
Z},g|_{U\setminus Z}).
\end{multline*}
\end{itemize}
We shall also refer to the latter relation as the {\it additivity
property for supports}. In other words, it says that a framed
correspondence in ${\bb Z}F_n(X,Y)$ whose support is a disjoint
union $Z\sqcup Z'$ equals the sum of the framed correspondences with
supports $Z$ and $Z'$ respectively. Note that ${\bb Z}F_n(X,Y)$ is
$\mathbb{Z}[Fr_n(X,Y)]$ modulo the subgroup generated by the
elements as above, because $0_n=0_n+0_n$ in this quotient group,
hence $0_n$ equals zero. Indeed, it is enough to observe that the
support of $0_n$ equals $\emptyset\sqcup\emptyset$ and then apply
the above relation to this support.

The elements of ${\bb Z}F_n(X,Y)$ are called {\it linear framed
correspondences of level $n$} or just {\it linear framed
correspondences}. It is worth to mention that ${\bb Z}F_n(X,Y)$
is a free abelian group generated by
the elements of $Fr_n(X,Y)$ with connected support.

Denote by ${\bb Z}F_*(k)$ an additive category whose objects are
those of $Sm/k$ with $\text{Hom}$-groups defined as
   $$\text{Hom}_{{\bb Z}F_*(k)}(X,Y)=\bigoplus_{n\geq 0}{\bb Z}F_n(X,Y).$$
The composition is induced by the composition in the category
$Fr_+(k)$.

There is a canonical functor $Sm/k \to {\bb Z}F_*(k)$ which is the identity on
objects and which takes a regular morphism $f: X\to Y$ to the linear
framed correspondence $1\cdot(X,X\times \bb A^0,pr_{\bb A^0},f\circ pr_X)
\in {\bb Z}F_0(k)$.

}\end{defs}

%\begin{defs}\label{d:boxpairing_linear}{\rm
%Let $X,Y,S$ and $T$ be schemes. The external product from
%Definition~\ref{d:boxpairing} induces a unique external product
%\[
%{\bb Z}F_n(X,Y)\times{\bb Z}F_m(S,T) \xrightarrow{-\boxtimes -} {\bb
%Z}F_{n+m}(X\times S, Y\times T)
%\]
%such that for any elements $a \in Fr_n(X,Y)$ and $b\in Fr_m(S,T)$
%one has $1\cdot a\boxtimes 1\cdot b=1\cdot(a\boxtimes b)\in {\bb
%Z}F_{n+m}(X\times S, Y\times T)$.
%}\end{defs}

%For the constant morphism $c\colon \bb A^1\to pt$, we set
%\[\Sigma:=-\boxtimes 1\cdot(\{0\},\bb A^1,t,c)\colon{\bb Z}F_n(X,Y)\to {\bb Z}F_{n+1}(X,Y)\]
%and refer to it as the \textit{suspension}.

%\begin{defs}\label{d:stable_linear_fr_corr}{\rm
%For any $k$-smooth variety $Y$ there is a presheaf ${\bb Z}F_*(-,Y)$
%on the category ${\bb Z}F_*(k)$ represented by $Y$. We also have a
%${\bb Z}F_*(k)$-presheaf
%   $${\bb Z}F(-,Y):=\colim({\bb Z}F_0(-,Y)\xrightarrow{\Sigma}{\bb Z}F_1(-,Y)
%     \xrightarrow{\Sigma}\cdots\xrightarrow{\Sigma}{\bb Z}F_n(-,Y)\xrightarrow{\Sigma}\cdots).$$
%For a $k$-smooth variety $X$, the elements of ${\bb Z}F(X,Y)$ are
%also called {\it stable linear framed correspondences}. Stable
%linear framed correspondences {\it do not form} morphisms of a
%category.}
%\end{defs}

\begin{rem}{\rm
For any $X,Y$ in $Sm/k$ one has equalities ${\bb Z}F_*(-,X\sqcup Y)={\bb
Z}F_*(-,X)\oplus{\bb Z}F_*(-,Y)$.
%and ${\bb Z}F(-,X\sqcup Y)={\bb
%Z}F(-,X)\oplus{\bb Z}F(-,Y)$.
}\end{rem}

\section{Two major theorems}
The main aim of this section is to state two major theorems
(Theorems \ref{Few_On_P_w_Tr} and \ref{Surj_Etale_exc}) on
presheaves with framed transfers.
This Section is copied mainly from \cite[Section 3]{GP4}.
{\it However Theorems \ref{Few_On_P_w_Tr} and \ref{Surj_Etale_exc}
are proved in \cite[Section 3]{GP4}
only if the characteristic of the infinite base field is not 2.
}
%As an application, we deduce the
%following result (which is the first assertion of
%Theorem~\ref{thm1}).

%\begin{thm}\label{Nisnevich_sheaf}
%For any $\bb A^1$-invariant quasi-stable $\bb ZF_*$-presheaf of
%abelian groups $\cc F$, the associated Nisnevich sheaf $\cc
%F_{\nis}$ is $\bb A^1$-invariant and quasi-stable if the
%characteristic of the base field $k$ is different from 2. If the
%characteristic of $k$ equals 2 and $\cc F$ is an $\bb A^1$-invariant
%quasi-stable $\bb ZF_*$-presheaf of $\bb Z[1/2]$-modules, then the
%associated Nisnevich sheaf $\cc F_{\nis}$ is $\bb A^1$-invariant and
%quasi-stable presheaf of $\bb Z[1/2]$-modules.
%\end{thm}

We need some definitions. We will write $(V,\varphi;g)$ for an
element $a$ in $Fr_n(X,Y)$. We also write $Z_a$ to denote the
support of $(V,\varphi;g)$. It is a closed subset in $X\times \bb
A^n$ which is finite over $X$ and which coincides with the common
vanishing locus of the functions $\varphi_1, ... , \varphi_n$ in
$V$. Next, by $\langle V,\varphi;g \rangle$ we denote the image of
the element $1\cdot(V,\varphi;g)$ in $\bb ZF_n(X,Y)$.

\begin{defs}
\label{sigma} Given any $k$-smooth variety $X$, there is a
distinguished morphism $\sigma_X=(X\times \bb A^1,t,pr_X)\in
Fr_1(X,X)$. Each morphism $f: Y\to X$ in $Sm/k$ can be regarded
tautologically as a morphism in $Fr_0(Y,X)$.
\end{defs}

In what follows by $SmOp/k$ we mean a category whose objects
are pairs $(X,V)$, where $X\in Sm/k$ and $V$ is an open subset of
$X$, with obvious morphisms of pairs.

\begin{defs}
\label{bar_ZF} Define $\bb ZF^{pr}_*(k)$ as an additive category
whose objects are those of $SmOp/k$
%pairs $(X,X^0)$ where
%$X \in Sm/k$ and $X^0 \subset X$ is an open subvariety.
and Hom-groups are defined as follows. We set for every
$n\geq 0$ and $(X,V),(Y,W)\in SmOp/k$:
$$\bb ZF^{pr}_*((Y,W),(X,V))=\ker[\bb ZF_n(Y,X)\oplus \bb ZF_n(W,V)\xrightarrow{i^*_Y-i_{X,*}} \bb ZF_n(W,X)],$$
where $i_Y: W \to Y$ is the embedding and $i_X: V \to X$ is the
embedding. In other words, the group $\bb ZF^{pr}_*((Y,W),(X,V))$
consists of pairs $(a,b) \in \bb ZF_n(Y,X)\oplus \bb ZF_n(Y^0,X^0)$
such that $i_Y\circ b= a\circ i_X$. By definition, the composite
$(a,b)\circ (a',b')$ is the pair $((a\circ b), (a'\circ b'))$.

We define $\overline {\bb ZF}_*(k)$ as an additive category whose
objects are those of $Sm/k$ and Hom-groups are defined as follows.
We set for every $n\geq 0$ and $X,Y\in Sm/k$:
$$\overline {\bb ZF}_*(Y,X)= \coker[\bb ZF_*(\bb A^1\times Y,X) \xrightarrow{i^*_0-i^*_1} \bb ZF_*(Y,X)].$$

Next, one defines $\overline {\bb ZF}^{pr}_*(k)$ as an additive
category whose objects are those of $SmOp/k$
%pairs $(X,X^0)$ where
%$X \in Sm/k$ and $X^0 \subset X$ is an open subvariety.
and Hom-groups are defined as follows. We set for every
$n\geq 0$ and $(X,V),(Y,W)\in SmOp/k$:
$$\overline {\bb ZF}^{pr}_*((Y,W),(X,V))= \coker[\bb ZF^{pr}_*(\bb A^1\times (Y,W),(X,V)) \xrightarrow{i^*_0-i^*_1} \bb ZF^{pr}_*((Y,W),(X,V)].$$
\end{defs}

\begin{notn}
Given $a \in \bb ZF_*(Y,X)$, denote by $[a]$ its class in $\overline
{\bb ZF}_*(Y,X)$. Similarly, if $r=(a,b) \in \bb
ZF^{pr}_*((Y,W),(X,V))$, then we will write $[[r]]$ to denote its
class in $\overline {\bb ZF}^{pr}_*((Y,W),(X,V))$.

If $X^0$ is open in $X$ and $Y^0$ is open in $Y$ and $g(Z^0) \subset
Y^0$ with $Z^0$ the support of $(V,\varphi;g)$, then $\langle\langle
V,\varphi;g \rangle\rangle$ will stand for the element $(\langle
V,\varphi;g \rangle,\langle V^0,\varphi^0;g^0 \rangle)$ in $\bb
ZF_n((X,X^0),(Y,Y^0))$.

We will as well write $[V,\varphi;g]$ to denote the class of
$\langle V,\varphi;g \rangle$ in $\overline {\bb ZF}_n(X,Y)$. In
turn, $[[V,\varphi;g]]$ will stand for the class of $\langle\langle
V,\varphi;g \rangle\rangle$ of $\overline {\bb
ZF}_n((X,X^0),(Y,Y^0))$.
\end{notn}

\begin{rem} Clearly, the category
$\bb ZF_*(k)$ is a full subcategory of $\bb ZF^{pr}_*(k)$ via the
assignment $X \mapsto (X,\emptyset)$. Similarly, the category
$\overline {\bb ZF}_*(k)$ is a full subcategory of $\overline {\bb
ZF}^{pr}_*(k)$ via the assignment $X \mapsto (X,\emptyset)$.
\end{rem}

In what follows we will also use the following category.

\begin{defs}
Let $\overline{\overline {\bb ZF}}_*(k)$ be a category whose objects
are those of $SmOp/k$ and whose $Hom$-groups are obtained from the
$Hom$-groups of the category $\overline {\bb ZF}^{pr}_*(k)$ by
annihilating the identity morphisms $id_{(X,X)}$ of objects of the
form $(X,X)$ for all $X \in Sm/k$.
\end{defs}

\begin{notn}
\label{sigma_and_pairs} If $r=(a,b) \in \bb ZF^{pr}_*((Y,W),(X,V))$,
then we will write $\overline {[[r]]}$ for its class in
$\overline{\overline {\bb ZF}}_*((Y,W),(X,V))$. For $(X,V)$ in
$SmOp/k$ we write $\langle\langle \sigma_X \rangle\rangle$ for the
morphism $(1\cdot \sigma_X, 1\cdot \sigma_V)$ in $\bb
ZF_1((X,V),(X,V))$.

We will denote by $\overline {[[V,\varphi;g]]}$ the class of the
element $[[V,\varphi;g]]$ in $\overline{\overline {\bb
ZF}}_n((X,X^0),(Y,Y^0))$.
\end{notn}

\begin{constr}
\label{F_on_pairs}
Let $\cc F$ be an $\bb A^1$-invariant $\bb ZF_*$-presheaf of abelian
groups. Then the assignments
$(X,V) \mapsto \cc F(X,V):= \cc F(V)/Im(\cc F(X))$
and
$$(a,b)\mapsto [(a,b)^*=b^*: \cc F(V)/Im(\cc F(X)) \to \cc F(W)/Im(\cc F(Y))],$$
for any $(a,b)\in \bb ZF_*((Y,W),(X,V))$ define a presheaf $\cc
F^{pairs}$ on the category $\overline{\overline {\bb ZF}}_*(k)$.
\end{constr}

To formulate
%further two theorems relating
the surjective part of the \'{e}tale excision
property, we need some preparations. Let
$S\subset X$ and $S'\subset X'$ be closed subsets. Let
$$\xymatrix{V'\ar[r]\ar[d]&X'\ar^{\Pi}[d]\\
               V\ar[r]&X}$$
be an elementary distinguished square with $X$ and $X'$ affine
$k$-smooth. Let $S=X-V$ and $S'=X'-V'$ be closed subschemes equipped
with reduced structures. Let $x\in S$ and $x' \in S'$ be two points
such that $\Pi(x')=x$. Let $U=Spec(\cc O_{X,x})$ and $U'=Spec(\cc
O_{X',x'})$. Let $\pi: U' \to U$ be the morphism induced by $\Pi$.

%\begin{thm}[Injective \'{e}tale excision]
%\label{Inj_Etale_exc} Under the notation above there is an integer
%$N$ and a morphism $r \in\bb ZF_N((U,U-S),(X',X'-S'))$ such that
%$$\overline {[[\Pi]]}\circ \overline {[[r]]}=\overline{[[can]]}\circ \overline{[[\sigma^N_U]]} $$
%in $\overline {\overline {\bb ZF_N}}((U,U-S),(X,X-S))$,
%where $can: U \to X$ is the canonical morphism.
%\end{thm}

%The statements of the next theorem depend on the characteristic of the base field $k$.

\begin{thm}[Surjective \'{e}tale excision]
\label{Surj_Etale_exc} Under the above notation
%suppose in addition that $S$ is $k$-smooth
%and $k$ is of characteristic different from 2. Then
there are an integer
$N$ and a morphism $l \in \bb ZF_N((U,U-S),(X',X'-S'))$ such that
$$\overline {[[l]]}\circ \overline {[[\pi]]}=\overline{[[can']]}\circ \overline{[[\sigma^N_{U'}]]} $$
in $\overline {\overline {\bb ZF_N}}((U',U'-S'),(X',X'-S'))$ with
$can': U' \to X'$ the canonical morphism.

%If the charcteristic of $k$ is $2$, then
%there are an integer
%$N$ and a morphism $l \in \bb ZF_N((U,U-S),(X',X'-S'))$ such that
%$$2\cdot \overline {[[l]]}\circ \overline {[[\pi]]}=2\cdot \overline{[[can']]}\circ \overline{[[\sigma^N_{U'}]]} $$
%in $\overline {\overline {\bb ZF_N}}((U',U'-S'),(X',X'-S'))$.
\end{thm}

We are now in a position to prove the following

\begin{thm}
\label{Few_On_P_w_Tr} For any $\bb A^1$-invariant quasi-stable $\bb
ZF_*$-presheaf of abelian groups $\cc F$
%the following statements
%are true:
%\begin{itemize}
%\ref{Surj_Etale_exc}
the map
$$[[\Pi]]^*: \cc F(U-S)/\cc F(U) \to  \cc F(U'-S')/\cc F(U')$$
is an
surjective.
%\end{itemize}
\end{thm}

\begin{proof}
Without loss of generality we may assume that $\cc F$ is stable and prove the theorem for this case
(by slight modifications, which are left to the reader, the theorem is similarly proved if $\cc F$ is quasi-stable).
%Assertions (1), (3) and (3') follow from Theorems \ref{InjOnAffLine} and
%\ref{Local_Inj}.
To prove assertion
%(2), (4) and
(5), use
Construction \ref{F_on_pairs} and Theorem
%and apply Theorems \ref{Ex_on_line},
%\ref{Exc_On_Rel_Aff_Line}, \ref{Inj_Etale_exc}, and
\ref{Surj_Etale_exc}
%respectively
(recall that $\cc F$ is stable).
\end{proof}

\section{Notation and agreements}
\label{Notation_Agreements}
{\it This Section is a copy paste of \cite[Section 3]{GP4}. We include it here
for convenience of the reader.
}
\begin{notn}
\label{Main} Given a morphism $a \in Fr_n(Y,X)$, we will write
$\langle a\rangle$ for the image of $1\cdot a$ in $\bb ZF_n(Y,X)$ and write
$[a]$ for the class of $\langle a\rangle$ in $\overline {\bb
ZF}_n(Y,X)$.

Given a morphism $a \in Fr_n(Y,X)$, we will write $Z_a$ for the
support of $a$ (it is a closed subset in $Y\times \bb A^n$ which
finite over $Y$ and determined by $a$ uniquely). Also, we will often
write
$$(\cc V_a,\varphi_a: \cc V_a \to \bb A^n; g_a: \cc V_a\to X) \ \text{or shorter} \ (\cc V_a,\varphi_a; g_a) $$
for a representative of the morphism $a$ (here $(\cc V_a, \rho: \cc
V_a \to Y\times \bb A^n, s: Z_a \hookrightarrow \cc V_a)$ is an
\'{e}tale neighborhood of $Z_a$ in $Y\times \bb A^n$).
\end{notn}

\begin{lem}
\label{Disjoint_Support} If the support $Z_a$ of an element $a=(\cc
V,\varphi; g) \in Fr_n(X,Y)$ is a disjoint union of $Z_1$ and
$Z_2$, then the element $a$ determines two elements $a_1$ and $a_2$
in $Fr_n(X,Y)$. Namely, $a_1=(\cc V-Z_2,\varphi|_{\cc V-Z_2};
g|_{\cc V-Z_2})$ and $a_2=(\cc V-Z_1,\varphi|_{\cc V-Z_1}; g|_{\cc
V-Z_1})$. Moreover, by the definition of $\bb ZF_n(X,Y)$ one has the
equality
$$\langle a \rangle= \langle a_1 \rangle + \langle a_2 \rangle$$
in $\bb ZF_n(X,Y)$.
\end{lem}

\begin{defs}
\label{Runs_Inside} Let $i_Y: Y'\hookrightarrow Y$ and $i_X:
X'\hookrightarrow X$ be open embeddings. Let $a\in Fr_n(Y,X)$. We
say that the restriction $a|_{Y'}$ of $a$ to $Y'$ runs inside $X'$,
if there is $a' \in Fr_n(Y',X')$ such that
\begin{equation}
\label{Eq_Runs_Inside}
i_X\circ a'=a\circ i_Y
\end{equation}
in $Fr_n(Y',X)$.

It is easy to see that if there is a morphism $a'$ satisfying
condition (\ref{Eq_Runs_Inside}), then it is unique. In this case
the pair $(a,a')$ is an element of $\bb ZF_n((Y,Y'),(X,X'))$. For
brevity we will write $\langle\langle a\rangle\rangle$
for $(a,a') \in \bb ZF_n((Y,Y'),(X,X'))$
and write $[[a]]$ to denote the class of $\langle\langle a\rangle\rangle$
in $\overline {\bb ZF}_n((Y,Y'),(X,X'))$.
\end{defs}

\begin{lem}
\label{Criteria} Let $i_Y: Y'\hookrightarrow Y$ and $i_X:
X'\hookrightarrow X$ be open embeddings. Let $a\in Fr_n(Y,X)$. Let
$Z_a \subset Y\times \bb A^n$ be the support of $a$. Set $Z'_a=Z_a
\cap Y' \times \bb A^n$. Then the following are equivalent:
\begin{itemize}
\item[(1)]
$g_a(Z'_a)\subset X'$;
\item[(2)]
the morphism $a|_{Y'}$ runs inside $X'$.
\end{itemize}
\end{lem}

\begin{proof}
$(1) \Rightarrow (2)$. Set $\cc V'=p^{-1}_Y \cap g^{-1}(X')$, where
$p_Y=pr_Y\circ \rho_a: \cc V \to Y\times \bb A^n$. Then $a':=(\cc
V', \varphi|_{\cc V'}); g|_{\cc V'})\in Fr_n(Y',X')$ satisfies
condition (\ref{Eq_Runs_Inside}).

$(2) \Rightarrow (1)$. If $a|_{Y'}$ runs inside $X'$, then for some
$a'=(\cc V', \varphi'; g') \in Fr_n(Y',X'))$ equality
(\ref{Eq_Runs_Inside}) holds. In this case the support $Z'$ of $a'$
must coincide with $Z'_a=Z_a \cap Y' \times \bb A^n$ and
$g_a|_{Z'}=g'|_{Z'}$. Since $g'(Z')$ is a subset of $X'$, then
$g_a(Z'_a)=g_a(Z')\subset X'$.
\end{proof}

\begin{cor}
\label{Critaria_for_H} Let $i_Y: Y'\hookrightarrow Y$ and $i_X:
X'\hookrightarrow X$ be open embeddings. Let $h_{\theta}=(\cc
V_{\theta},\varphi_{\theta}; g_{\theta})\in Fr_n(\bb A^1\times
Y,X)$. Suppose $Z_{\theta}$, the support of $h_{\theta}$, is such
that for $Z'_{\theta}:=Z_{\theta}\cap \bb A^1\times Y' \times \bb
A^n$ one has $g_{\theta}(Z'_{\theta})\subset X'$. Then there are
morphisms $\langle\langle h_{\theta} \rangle\rangle \in\bb ZF_n(\bb
A^1\times (Y,Y'),(X,X'))$, $\langle\langle h_0 \rangle\rangle \in\bb
ZF_n((Y,Y'),(X,X'))$, $\langle\langle h_1 \rangle\rangle \in\bb
ZF_n((Y,Y'),(X,X'))$ and one has an obvious equality
$$[[h_0]=[[h_1]]$$
in $\overline {\bb ZF}_n((Y,Y'),(X,X'))$.
\end{cor}

\begin{lem}[A disconnected support case]
\label{Criteria_for_pairs} Let $i_Y: Y'\hookrightarrow Y$ and $i_X:
X'\hookrightarrow X$ be open embeddings. Let $a\in Fr_n(Y,X)$ and
let $Z_a \subset Y\times \bb A^n$ be the support of $a$. Set
$Z'_a=Z_a \cap Y' \times \bb A^n$. Suppose that $Z_a=Z_{a,1} \sqcup
Z_{a,2}$. For $i=1,2$ set $\cc V_i=\cc V_a-Z_{a,j}$ with $j\in
\{1,2\}$ and $j\neq i$. Also set $\varphi_i=\varphi_a|_{\cc V_i}$
and $g_i=g_a|_{\cc V_i}$. Suppose $a|_{Y'}$ runs inside $X'$, then
\begin{itemize}
\item[(1)]
for each $i=1,2$ the morphism $a_i:=(\cc V_i, \varphi_i; g_i)$ is such that $a_i|_{Y'}$ runs inside $X'$;
\item[(2)]
$\langle\langle a \rangle\rangle = \langle\langle a_1 \rangle\rangle + \langle\langle a_2 \rangle\rangle$
in $\bb ZF_n((Y,Y'),(X,X'))$.
\end{itemize}

\end{lem}

\section{Surjectivity of the \'{e}tale excision}
\label{section_surjectivity}
Let
\begin{equation}\label{Nisn_square}
\xymatrix{V'\ar[r]\ar[d]&X'\ar^{\Pi}[d]\\
               V\ar[r]&X}
\end{equation}
be an elementary distinguished square with affine $k$-smooth $X$ and $X'$.
Let $S=X-V$ and $S'=X'-V'$ be closed subschemes equipped with reduced structures.
Let $x\in S$ and $x' \in S'$ be two closed points such that $\Pi(x')=x$.
Let
$U=Spec(\cc O_{X,x})$
and
$U'=Spec(\cc O_{X',x'})$.
Let $\pi_U: U' \to U$ be the morphism induced by $\Pi$.
To prove Theorem \ref{Surj_Etale_exc}, it suffices to find morphisms
$a\in \bb ZF_N((U,U-S)),(X',X'-S'))$ and $b_G\in \bb
ZF_N((U',U'-S')),(X'-S',X'-S'))$ such that
\begin{equation}
\label{Surjectivity_Part_Details}
[[a]]\circ [[\pi_U]]-[[j]]\circ [[b_G]] =[[can']]\circ [[\sigma^N_U]]
\end{equation}
in $\overline {\bb ZF}_N(U',U'-S')),(X',X'-S'))$. Here $j: (X'-S',X'-S') \to
(X',X'-S')$ and $can': (U',U'-S') \to (X',X'-S')$ are inclusions.
%In this section we do some preparations to
%construct the desired morphisms
%$a\in \bb ZF_N((U,U-S)),(X',X'-S'))$ and $b_G\in \bb
%ZF_N((U,U-S)),(X-S,X-S))$ in Section~\ref{Proof_of_Etale_Injectivity}
%satisfying~\eqref{Injectivity_Part_Details}.
Shrinking $X$ and $X'$ such that $x$ is in the shrunk $X$ and $x'$ is in the shrunk $X'$ one can find a commutative diagram of the form
(see the diagram (\ref{SquareDiagram_new_2}))
\begin{equation}
\label{SquareDiagram_new}
    \xymatrix{
    X \ar[drrrr]_{q} && X'\ar[ll]_{\Pi} \ar[drr]^{q'}\ar[rr]^{j'}&& \overline X'\ar[d]_{\overline q'} && X'_{\infty}\ar[ll]_{i'}\ar[lld]_{q'_{\infty}} &\\
     &&&& B  &\\,   }
\end{equation}
with $B$-smooth affine $B$-schemes $X'$and $X$,
where $B$ is an affine open in $\bb P^{n-1}$, $q: X\to B$ is an almost elementary fibration
in the sense of
\cite[Definition 2.1]{PSV}
such that $\omega_{B/k}\cong \cc O_B$, $\omega_{X/k}\cong \cc O_X$,
$q'_{\infty}: X'_{\infty}\to B$ is finite surjective, $i'$ is a closed embedding of schemes, $j'$ is an open embedding of schemes,
the closed subscheme $X'_{\infty}$ of the scheme $\bar X'$
{\it is an ample
Cartier divisor in
}
$\bar X'$
and $X'=\bar X'-X'_{\infty}$.
Also, $S'\subset X'$, $S\subset X$ and $\Pi|_{S'}: S'\to S$ is a scheme isomorphism and $S'$ is finite over $B$.
Particularly, $S'$ is closed in $\bar X'$ and $S'\cap X'_{\infty}=\emptyset$.

Since $\Pi|_{S'}: S'\to S$ is a scheme isomorphism, hence $q|_{S}: S\to B$ is finite.
The morphism $q': X'\to X$ is a smooth morphism of relative dimention one, since
$q$ has the same properties and $\Pi$ is \'{e}tale.

%\begin{notn}
%Set $X'_{\infty}=Y'\sqcup (\bar {\Pi})^{-1}(X_{\infty})\subset \bar X'$. It is a Cartier divisor on $\bar X'$.
%One has an equality
%$X'=\bar X' - X'_{\infty}$.
%Note that $X'_{\infty}$ is finite over $B$.
%We know already that $S'$ is finite over $B$, closed in $\bar X'$ and
%$S'\cap (Y'\sqcup (\bar {\Pi})^{-1}(X_{\infty}))=\emptyset$.
%\end{notn}

\begin{rem}
\label{Stable_Normal_Bundle}
Let $q: X\to B$ be the almost elementary fibration from
%Remark \ref{Elementary_fibr},
the diagram (\ref{SquareDiagram_new}),
then
$\Omega^1_{X/B}\cong \cc O_X$.
In fact, $\omega_{X/k}\cong q^*(\omega_{B/k})\otimes \omega_{X/B}$.
Thus $\omega_{X/B}\cong \cc O_X$. Since $X/B$ is a smooth relative curve,
then $\Omega^1_{X/B}=\omega_{X/B}\cong \cc O_X$. Since
$\Pi: X'\to X$ is \'{e}tale, hence $\Omega^1_{X'/B}\cong \cc O_{X'}$.

Since $X'$ is an affine $B$-scheme and $B$ is affine there is a closed embedding
$in: X'\hookrightarrow B\times \bb A^N$
%is a closed
of $B$-schemes. Then one has
$[\cc N(in)]=(N-1)[\cc O_{X'}]$ in $K_0(X')$, where $\cc N(in)$ is the normal bundle to $X'$
for the imbedding $in$.

Thus by increasing the integer $N$, we may and will assume that the normal
bundle $\cc N(in)$ is isomorphic to the trivial bundle $\cc
O^{N-1}_{X'}$.
\end{rem}
Using the indicated properties
%(i)--(iv)
of the diagram
(\ref{SquareDiagram_new}),
Remark \ref{Stable_Normal_Bundle}
and
repeating arguments from the proof of
\cite[Lemma 8.2]{GP4}
we get the following
\begin{prop}
\label{Framing_Of_X'}
Let $B$, $X'$ and %(edited) X is changed to X'
$q: X'\to B$ be the schemes and the morphism defined just above the diagram
(\ref{SquareDiagram_new}). Then there are an
integer $N\geq 0$, a closed embedding $X' \hookrightarrow B\times \bb
A^N$ of $B$-schemes, an \'{e}tale affine neighborhood $(\cc V, \rho:
\cc V \to B\times \bb A^N, s: X' \hookrightarrow \cc V)$ of $X'$ in
$B\times \bb A^N$, functions $\varphi_1,...,\varphi_{N-1} \in k[\cc
V]$ and a morphism $r: \cc V \to X'$ such that:
\begin{itemize}
\item[(i)]
the functions $\varphi_1,...,\varphi_{N-1}$ generate the ideal
$I_{s(X')}$ in $k[\cc V]$ defining the closed subscheme $s(X')$ of
$\cc V$;
\item[(ii)]
$r\circ s = id_{X'}$;
\item[(iii)]
the morphism $r$ is a $B$-scheme morphism if $\cc V$ is regarded as
a $B$-scheme via the morphism $pr_U\circ \rho$, and $X'$ is regarded
as a $B$-scheme via the morphism $q'$.
\end{itemize}
\end{prop}

Let $b=q(x)\in B$. It is a closed point of $B$. Let $S'_b$ and $S_b$
be the fibres over the point $b$ of the morphisms $q'|_{S'}: S'\to B$ and $q|_S: S\to B$
respectively. The schemes $S'_b$ and $S_b$ consists of finitely many closed points since
$S'$ and $S$ are finite over B. Let $s'_0,...,s'_n$ be the set of points of $S'_b$ (here $s'_0=x'$).
Let $s_i=\Pi(s'_i)\in S_b$. Since $\Pi|_{S'}: S'\to S$ is a scheme isomorphism,
hence $s_0,...,s_n$ be the set of points of $S_b$ and $s_0=x$.
\begin{defs}\label{W_and_W'}
Set $W'=Spec(\cc O_{X',\{s'_0,...,s'_n\}})$, $W=Spec(\cc O_{X,\{s_0,...,s_n\}})$.
%Let $x\in S$, $x'\in S'$ be such that $\Pi(x')=x$.
%Set $U=Spec(\cc O_{X,x})$,
%$\cc X=U\times_B X$.
%$\cc D=U\times_B D$.
%There is an obvious morphism $\Delta=(id,can): U\to U\times_B X$. It
%is a section of the projection $p_U: U\times_B X \to U$. Let $p_X:
%U\times_B X \to X$ be the projection onto $X$.
Let $\pi_W: W'\to W$ be
the restriction of $\Pi$ to $W'$.
Note that $S'$ (respectively $S$) is a closed subset of $W'$ (respectively of $W$).
\end{defs}

\begin{notn}
\label{X'_script} In what follows we will write
%$U\times X$ to denote
%$U\times_B X$,
%$U\times D$ for $U\times_B D$,
$W\times X'$  to denote $W\times_B X'$, $W'\times X'$
%{\bf to denote
%$W'\times_B X'$, $S'_b\times S'_b$
%to denote
%$S'_b\times_B S'_b$
etc.
Here $X'$ is regarded as a $B$-scheme via the
morphism $q'=q\circ \Pi$.
We will write $\Delta: W'\hookrightarrow W'\times X'$ for the diagonal embedding.

%{\bf The scheme $S'_b\times S'_b$ is a closed subscheme in $W'\times X'$. Its support is a finite subset $T$ of
%closed points of \ $W'\times X'$. Clearly,
%}
%$$T=Diag\sqcup (T-Diag),$$
%where $Diag=\Delta(\{s'_0,...,s'_n\})$ is the set of closed points of the closed subscheme
%$\Delta(W') \subseteq W'\times X'$.
If $J\subset k[W'\times X']$ is the ideal defining the closed subset $\Delta(W')$ in $W'\times X':=W'\times_B X'$,
then set $\Delta^{(2)}(W')=Spec(k[W'\times X']/J^2)$.

The ideal $J$ is locally principal. Thus, the ideal $J^2$ is locally principal too.
\end{notn}

\begin{lem}\label{e'}
Under the notation of Proposition \ref{Framing_Of_X'} the following is true:
\begin{itemize}
\item[(i)]
For any $e'\in J$ the element $(id_{W'}\times r)^*(e')\in k[W'\times \cc V]$ is in the ideal
$I_{(id\times s)(W')}\subset k[W'\times \cc V]$ defining the closed subset $(id_{W'},s)(W')$ of $W'\times \cc V$;
\item[(ii)]
Let $pr_{\cc V}: W'\times \cc V\to \cc V$ be the projection. Then for any $i=1,...,N-1$ for
$\phi'_i:=pr^*_{\cc V}{\phi_i}$ one has: $\phi'_i\in I_{W'\times X'}\subset I_{(id\times s)(W')}$;
\item[(iii)]
%Let $\bar {\phi}'_i=\phi'_i \ \text{mod} \ I^2_{(id\times s)(W')}$.
There exists an element
$e'\in J\subset k[W'\times X']$ such that $e' \ \text{mod}\ J^2$
is a free generator of the
$k[W']$-module $J/J^2$.
For any such element
$e'\in J\subset k[W'\times X']$
the elements
\begin{equation}\label{basis_e_prime_phi}
(id_{W'}\times r)^*(e'), {\phi}'_1,...,{\phi}'_{N-1} \ \text{mod} \ I^2_{(id\times s)(W')}
\end{equation}
form a free basis of the $k[W']$-module $I_{(id\times s)(W')}/I^2_{(id\times s)(W')}$;
\item[(iv)]
Let
$t'_i=t_i-p^*_{W'}(t_i|_{W'})\in k[W'\times \bb A^N]$.
Then
\begin{equation}\label{basis_t_prime}
(id\times \rho)^*(t'_1),...,(id\times \rho)^*(t'_N) \ \text{mod} \ I^2_{(id\times s)(W')}
\end{equation}
is a free basis of the $k[W']$-module
$I_{(id\times s)(W')}/I^2_{(id\times s)(W')}$;
\item[(v)]
For any $e'\in J\subset k[W'\times X']$
as in the item (iii) of this lemma let $A(e')\in GL_N(k[W'])$ be a matrix transforming
the free basis (\ref{basis_t_prime})
%$$(id_{W'}\times r)^*(e'), {\phi}'_1,...,{\phi}'_{N-1} \ \text{mod} \ I^2_{(id\times s)(W')}$$
of the $k[W']$-module $I_{(id\times s)(W')}/I^2_{(id\times s)(W')}$
to the free basis (\ref{basis_e_prime_phi})
of the same $k[W']$-module.
%For any $e'\in J\subset k[W'\times X']$
%from the item (iii) of this lemma let $A(e')\in GL_N(k[W'])$ be such a matrix, which transforms
%the free basis
%$\bar {e}',\bar {\phi}'_1,...,\bar {\phi}'_{N-1} \ \text{mod} \ I^2_{(id\times s)(W')}$
%to the free basis
%$(id\times \rho)^*(t'_1),...,(id\times \rho)^*t'_N$.
Then there exists $e'\in J$ as in the item (iii) of this lemma such that $det(A(e'))=1$.
\end{itemize}
\end{lem}

\begin{proof}
The proof of this Lemma is straightforward.
{\it However the assertion (v) is of great utility for
the purpose of the present paper.
}
So, we derive it here from assertions (i)--(iv).
By the item (iii) of this Lemma there exists an element
$e'\in J\subset k[W'\times X']$ such that $e' \ \text{mod } J^2$
is a free generator of the
$k[W']$-module $J/J^2$.
Choose and fix any $e'\in J$ having this property.
By the item (iii) of this Lemma
the elements
(\ref{basis_e_prime_phi})
form a free basis of the $k[W']$-module $I_{(id\times s)(W')}/I^2_{(id\times s)(W')}$;
Let $A(e')\in GL_N(k[W'])$ be a matrix transforming
the free basis (\ref{basis_t_prime})
of the $k[W']$-module $I_{(id\times s)(W')}/I^2_{(id\times s)(W')}$
to the free basis (\ref{basis_e_prime_phi})
of the same $k[W']$-module.
Let $det(A(e'))\in k[W']^{\times}$ be its determinant.

Set $e'_{new}=det(A(e'))^{-1}\cdot e'$. Clearly, $det(A(e'_{new}))=1$.
\end{proof}

\begin{prop}
\label{h'_theta}
%Under the conditions of Remark \ref{Elementary_fibr}
%Notation
%\ref{Final_Notation}
Consider the diagram (\ref{SquareDiagram_new}).
Then under the notation \ref{X'_script}
there are functions
$F\in k[W\times X']$
and
$h_{\theta}\in k[\bb A^1\times W'\times X']$
($\theta$ is the parameter on the left factor $\bb A^1$)
such that the following properties hold for the functions
$h_{\theta}$, $h_1:=h_{\theta}|_{1\times W'\times X'}$ and $h_0:=h_{\theta}|_{0\times W'\times X'}$:
\begin{itemize}
\item[(a)]
the morphism $(pr,h_{\theta}): \bb A^1\times W'\times X' \to \bb
A^1\times W'\times \bb A^1$ is finite and surjective, hence the
closed subscheme $Z'_{\theta}:=(h_{\theta})^{-1}(0)\subset \bb
A^1\times W'\times X'$ is finite flat and surjective over $\bb A^1
\times W'$;
\item[(b)]
for the closed subscheme
$Z'_{0}:=(h_{0})^{-1}(0)$ one has
$Z'_{0}=\Delta(W') \sqcup G$ (an equality of closed subschemes) and $G\subset W'\times (X'-S')$;
\item[(c)]
$h_1=(\pi \times id_{X'})^{*}(F)$ (we write $Z'_1$ to denote the closed
subscheme $\{h_1=0\}$);
\item[(d)]
$Z'_{\theta}\cap \bb A^1\times (W'-S')\times S'=\emptyset$ or,
equivalently,\\ $Z'_{\theta}\cap \bb A^1\times (W'-S')\times X'
\subset \bb A^1 \times (W'-S') \times (X'-S')$;
\item[(e)]
the morphism $(pr_W,F): W\times X' \to W\times \bb A^1$ is finite
surjective, and hence the closed subscheme $Z_1:=F^{-1}(0)\subset
W\times X'$ is finite flat and surjective over $W$;
\item[(f)]
$Z_1\cap (W-S)\times S'=\emptyset$ or, equivalently, $Z_1\cap
(W-S)\times X' \subset (W-S)\times (X'-S')$;
\item[(g)] $h_0 \in J$, $h_0 \ \text{mod } J^2$
is a free generator of the
$k[W']$-module $J/J^2$ and $det((A(h_0))=1\in k[W']$, where
$A(h_0)$ is as in the item (v) of Lemma \ref{e'}.
\end{itemize}
\end{prop}
%{\bf TO GET THE PROPERTY $G\subset W'\times (X'-S')$ WE NEED TO REQUIRE THE FOLLOWING ADDITIONAL CONDITION ON THE FUNCTION
%$e'$ \\}
To prove Proposition
\ref{h'_theta}
we need in several lemmas.
\begin{lem}\label{Delta_2_S'}
Consider the graph $\Delta(S')\subset W'\times S':=W'\times_B S'$ of the embedding $S'\hookrightarrow W'$
from Definition \ref{W_and_W'}.
Then $\Delta(S')$ is a principal Cartier divisor on $W'\times S'$.
\end{lem}

Recall that $\Pi|_{S'}: S'\to S$ is a scheme isomorpism.
\begin{lem}\label{Gamma_2_S'}
Consider the section $((\pi_{W})|_{S'},id): S'\to W\times S'$ of the projection $W\times S'\to S'$.
Let $\Gamma$ be its image in $W\times S'$ (it is a closed subscheme in $W\times S'$).Then
$\Gamma$ is a principal Cartier divisor on $W\times S'$.
\end{lem}

\begin{notn}
Let $J_{\Delta(S')}\subset k[W'\times S']$ be the ideal defining $\Delta(S')$ in $W'\times S'$.
Let $J_{\Gamma}\subset k[W\times S']$ be the ideal defining $\Gamma$ in $W\times S'$.
Note that the scheme morphism
$\pi_W\times id: W'\times S'\to W\times S'$
identifies the closed subscheme $\Delta(S')$ with the closed subcheme $\Gamma$.

Set
$\Delta^{(2)}(S')=Spec(k[W'\times S']/J^2_{\Delta(S')})$
and
$\Gamma^{(2)}=Spec(k[W\times S']/J^2_{\Gamma})$.
Then $\pi_W\times id|_{\Delta^{(2)}(S')}$ identifies $\Delta^{(2)}(S')$ with the closed subcheme $\Gamma^{(2)}$.
Let $\tau: \Delta^{(2)}(S')\to \Gamma^{(2)}$ be this scheme isomorphism.
\end{notn}
For an open inclusion $T^{\circ} \subset T$ of affine schemes and an $\Gamma(T,\cc O_T)$-module $M$
we will write $M|_{T^{\circ}}$ for the $\Gamma(T^{\circ},\cc O_{T^{\circ}})$-module
$M\otimes_{\Gamma(T,\cc O_T)} \Gamma(T^{\circ},\cc O_{T^{\circ}})$.

By Lemma \ref{Delta_2_S'} the open subscheme $W'\times S'-\Delta(S')$ of the affine scheme
$W'\times S'$ is an affine scheme.
Let $\bar J=J\cdot k[W'\times S']\subset k[W'\times S']$.
Then it is easy to check that the inclusion
$\bar J|_{W'\times S'-\Delta(S')}\subset k[W'\times S'-\Delta(S')]$
is an equality. Thus, one has an inclusion of ideals
$\bar J\subset J_{\Delta(S')}$
of the $k$-algebra $k[W'\times S']$.
{\it The following lemma requires a proof.
}
However we left it to the reader.

\begin{lem}\label{bar_J_and_J_Delta_prime}
The inclusion $\bar J\subset J_{\Delta(S')}$ of the ideals is an equality.
As a consequence $\bar J/\bar J^2=J_{\Delta(S')}/J^2_{\Delta(S')}$.
\end{lem}

\begin{rem}
Lemma \ref{bar_J_and_J_Delta_prime} guaranties the following: if
$e'\in J$ is such that
$e' \ \text{mod} \ J^2$
is a free basis of the free rank one $k[W']$-module
$J/J^2$,
then the element
$e' \ \text{mod} \ J^2_{\Delta(S')}$
is a free basis of the free rank one $k[S']$-module
$J_{\Delta(S')}/J^2_{\Delta(S')}$.
Particularly, $e' \ \text{mod} \ J^2_{\Delta(S')}$  does not vanish.
\end{rem}
The proof of following lemma is straightforward.
\begin{lem}\label{J_Delta'_and_J_Gamma}
%$\tau^*: J_{\Gamma}\to J_{\Delta'}$ of the ideals.
Any $k[\Gamma]$-module $M$ we will regard as an $k[\Delta(S')]$-module using the $k$-algebra isomorphism
$(\tau)^*: k[\Gamma]\to k[\Delta(S')]$.
As indicated above the morphism $\pi_W\times id: W'\times S' \to W\times S'$ induces
the scheme isomorphism
$\tau: \Delta^{(2)}(S')\to \Gamma^{(2)}$
and an isomorphism
$J_{\Gamma}/J^2_{\Gamma} \xrightarrow{\tau^*} J_{\Delta'}/J^2_{\Delta'}$
of the $k[\Gamma]$-modules.

Moreover, $J_{\Gamma}/J^2_{\Gamma}$ and $J_{\Delta'}/J^2_{\Delta'}$
are free rank one $k[\Gamma]$-modules.

If $g\in J_{\Gamma}/J^2_{\Gamma}$ is a free generator of the free rank one $k[\Gamma]$-module
$J_{\Gamma}/J^2_{\Gamma}$, then any lift $\tilde g\in J_{\Gamma}$ of $g$ is a free generator
of the free rank one $k[W\times S']$-module $J_{\Gamma}$.

If $\tilde g\in J_{\Gamma}$ is a free generator of the $k[W\times S']$-module $J_{\Gamma}$,
then $(\pi_W\times id)^*(\tilde g)$ is a free generator of the $k[W'\times S']$-module $J_{\Delta'}$.
\end{lem}

\begin{lem}\label{e'_and_e}
There exists an element $e'\in J$ such that $e' \ \text{mod} \ J^2$ is a free basis of the $k[W']$-module $J/J^2$.
For any $e'\in J$ such that $e' \ \text{mod} \ J^2$ is a free basis of the $k[W']$-module $J/J^2$
there exists an element
$e\in J_{\Gamma}$ such that the following holds:
\begin{equation}\label{element_e}
e \in J_{\Gamma} \ \text{is a free basis of the free rank one} \ k[W\times S']-\text{module} \ J_{\Gamma},
\end{equation}
\begin{equation}\label{elements_e_and_e'}
(\pi_W\times id)^*(e)|_{\Delta^{(2)}(S')}=e'|_{\Delta^{(2)}(S')}.
\end{equation}

For any element $e \in J_{\Gamma}$ subjecting to condition
(\ref{element_e}) the element
$(\pi_W\times id)^*(e)$
is in $J_{\Delta(S')}$ and
it is a free basis of the free rank one
$k[W'\times S']$-module $J_{\Delta(S')}$.

%Moreover, one can find $e\in J_{\Gamma}$ subjecting as to condition
%(\ref{element_e}) so to the following one
%$$(\pi_W\times id)^*(e)|_{\Delta^{(2)}(S')}=e'|_{\Delta^{(2)}(S')}.$$
\end{lem}

\begin{proof}(of Lemma \ref{e'_and_e})
Let $e'\in J$ be such that $e' \ \text{mod} \ J^2$ is a free basis of the $k[W']$-module $J/J^2$.
let $\bar e'\in \bar J\subset k[W'\times S']$ be its image in $\bar J$.
By Lemma \ref{bar_J_and_J_Delta_prime} $\bar J=J_{\Delta(S')}$.
Thus $\bar e'\in J_{\Delta(S')}$ is such that
$e' \ \text{mod} \ J^2_{\Delta(S')}$
is a free basis of the $k[S']$-module
$J_{\Delta(S')}/J^2_{\Delta(S')}$.
The morphism
$\pi_W\times id: W'\times S' \to W\times S'$ induces
the scheme isomorphism
$\tau: \Delta^{(2)}(S')\to \Gamma^{(2)}$
as in Lemma \ref{J_Delta'_and_J_Gamma}
and an isomorphism
$J_{\Gamma}/J^2_{\Gamma} \to J_{\Delta(S')}/J^2_{\Delta(S')}$
of the $k[\Gamma]$-modules.
Set
$$\bar e=(\tau^{-1})^*(\bar {e}')\in k[\Gamma^{(2)}],$$
where $\tau^*$ is the pull-back map induced by the scheme isomorphism
$\tau$. It's clear now that $\bar e$ is a free basis of the $k[\Gamma]$-module
$J_{\Gamma}/J^2_{\Gamma}$. Let $e\in J_{\Gamma}$ be such that $\bar e= e \ \text{mod} \ J^2_{\Gamma}$.
By the third statement of Lemma \ref{J_Delta'_and_J_Gamma} the element
$e$
is a free basis of the free rank one $k[W\times S']$-module $J_{\Gamma}$.
The equality (\ref{elements_e_and_e'}) is true by the construction
of the element $e$. The first assertion of the lemma is proved.

The second one follows from the last statement of Lemma
\ref{J_Delta'_and_J_Gamma}.
The third assertion is the first part of the assertion (iii) of Lemma
\ref{e'}.
\end{proof}

\begin{proof}(of Proposition \ref{h'_theta})
Consider a $W'$-scheme $W'\times \bar X'$ and a $W$-scheme $W\times \bar X'$ and the morphism
$\pi_W\times id: W'\times \bar X'\to W\times_B \bar X'$.
Consider the Cartier divisors $W'\times X'_{\infty}$ and $W\times X'_{\infty}$
on $W'\times \bar X'$ and on $W\times \bar X'$ respectively.
%and $s'_{\infty}$ be its global section
%vanishing exactly on $W'\times \bar X'_{\infty}$.
Let $L$ be the line bundle on $W\times \bar X'$ corresponding to the Cartier divisor
$W\times X'_{\infty}$
and $s_{\infty}$ be its global section
vanishing exactly on $W\times \bar X'_\infty$.
Let $L'$ be the line bundle on $W'\times \bar X'$ corresponding to the Cartier divisor
$W'\times X'_{\infty}$.
Clearly, $(\pi_W\times id)^*(L)=L'$. The vanishing locus of the section
$s':=(\pi_W\times id)^*(s)$
is exactly the Cartier divisor
$W\times \bar X'_\infty$.

Let $e'\in J$ and $e\in J_{\Gamma}$ be as in Lemma \ref{e'_and_e}. Choose and fix a positive integer $M$ big enough.
Then there is a global section $s$ of the bundle $L^{\otimes M}$ such that\\
(i) $s|_{W\times \bar X'_\infty}$ has no zeros,\\
(ii) $s|_{W\times S'}=e\cdot (s^{\otimes M}_{\infty})|_{W\times S'}$.\\
Also, there is a global section $s'$ of the bundle $(L')^{\otimes M}$ such that \\
(i') $s'|_{W'\times \bar X'_{\infty}}=(\pi_W\times id)^*(s|_{W\times \bar X'_{\infty}})$,\\
(ii') $s'|_{W'\times S'}=(\pi_W\times id)^*(e)\cdot (s'_{\infty})^{\otimes M}|_{W'\times S'}$,\\
(iii') $s'|_{\Delta^{(2)}(W')}=e'\cdot (s'_{\infty})^{\otimes M}|_{\Delta^{(2)}(W')}$.\\
Set $F=(s|_{W\times X'})/(s^{\otimes M}_{\infty}|_{W\times X'})\in k[W\times X']$,\\
$F_1=(\pi_W\times id)^*(F)\in k[W'\times X']$,\\
$F_0=(s'|_{W'\times X'})/((s')^{\otimes M}_{\infty}|_{W'\times X'})\in k[W'\times X']$,\\
$h_{\theta}=(1-\theta)\cdot F_0 + \theta\cdot F_1\in k[\bb A^1\times W'\times X']$.\\
It is straightforward to check that
$F$ and $h_{\theta}$ subject to the conditions (a)--(f) of Proposition
\ref{h'_theta}.
Check for instance that $h_{\theta}$ subject to the condition (b).
The property (iii') of the section $s'$ yields that for the closed subscheme
$Z'_{0}:=(h_{0})^{-1}(0)$
from Proposition \ref{h'_theta}
one has
$Z'_{0}=\Delta(W') \sqcup G$ (a disjoint union of closed subschemes).
Since
$F_0|_{W'\times S'}=(\pi_W\times id)^*(e)$
and
$(\pi_W\times id)^*(e)$ is a free basis of
of the free rank one
$k[W'\times S']$-module $J_{\Delta(S')}$,
we have
$Z'_{0}\cap W'\times S'=\Delta(S')$. But also
$$Z'_{0}\cap W'\times S'=(G\sqcup \Delta(W'))\cap W'\times S'=(G\cap W'\times S')\sqcup (\Delta(W')\cap W'\times S')=
(G\cap W'\times S')\sqcup \Delta(S').$$
Thus, $G\cap W'\times S'=\emptyset$. Whence the condition (b).

If we take $e'\in J$ as in the item (v) of Lemma \ref{e'}
and for this specific $e'$ take $e\in J_{\Gamma}$ as in Lemma \ref{e'_and_e},
then the functions $F$ and $h_{\theta}$ subject to the condition (g) as well.
Proposition \ref{h'_theta} is proved.
%The proof of Theorem
%\ref{Surj_Etale_exc}
%is completed.
\end{proof}

Let $j: (X'-S',X'-S')\to (X',X'-S')$
be the obvious morphism,
$can'': W'\hookrightarrow X'$
be the inclusion and
$\pi_W: W'\to W$ be the restriction of $\Pi: X'\to X$ to $W'$.
We construct below in this Section certain morphisms
\begin{itemize}
\item[(i)]
$\alpha\in Fr_N(W,X')$,
%and $b_G\in \bb ZF_N((W',W'-S')),(X'-S',X'-S'))$
\item[(ii)]
$\beta_G\in Fr_N(W',X'-S')$
\end{itemize}
such that $\alpha|_{W-S}$ run inside $X'-S'$
in the sense of Definition \ref{Runs_Inside}.
The morphism $\beta_G|_{W'-S'}$ automatically run inside of $X'-S'$
in the sense of Definition \ref{Runs_Inside}.
The morphisms $\alpha$ and $\beta_G$ will be chosen
such that under the notation of Definition \ref{Runs_Inside} one has equality
\begin{equation}
\label{Surjectivity_Part_Details_2}
[[\alpha]]\circ [[\pi_W]]-[[j]]\circ [[\beta_G]]=[[can'']]\circ [[\sigma^N_W]] \in \overline {\bb ZF}_N((W',W'-S')),(X',X'-S')).
\end{equation}
%Let $inc: (U,U-S)\to (W,W-S)$ and $inc': (U',U'-S')\to (W',W'-S')$
%be the morphisms of pairs defined by the inclusions
Let $inc: U\hookrightarrow W$ and $inc': U'\hookrightarrow W'$ be the inclusions.
%respectively.
Set
$a=\alpha\circ inc\in Fr_N(U,X')$, $b_G=\beta_G \circ inc' \in Fr_N(U',X'-S')$.
Clearly, $a|_{W-S}$ run inside of $X'-S'$ and $b_G|_{W'-S'}$ run inside of $X'-S'$
in the sense of Definition \ref{Runs_Inside}.
The equality (\ref{Surjectivity_Part_Details_2}) yields an equality
$$[[a]]\circ [[\pi_U]]-[[j]]\circ [[b_G]]=[[can']]\circ [[\sigma^N_U]] \in \overline {\bb ZF}_N((U',U'-S')),(X',X'-S')),$$
where $can': U'\hookrightarrow X'$, $\pi_U: U'\to U$ are as at the beginning of this Section.
{\it Thus the morphisms
$[[a]]$ and $[[b_G]]$ are such that the equality
(\ref{Surjectivity_Part_Details}) does hold.
}
{\it To complete the proof of Theorem
\ref{Surj_Etale_exc} it remains
%to prove Proposition
%\ref{h'_theta}
%and
to construct
$$\alpha\in Fr_N(W,X') \ \text{and} \ \beta_G\in Fr_N(W',X')$$
satisfying the equality
(\ref{Surjectivity_Part_Details_2}).
}
%Firstly, assuming Proposition
%\ref{h'_theta}
%we construct morphisms
%$\alpha$ and $\beta_G$
%such that the equality
%(\ref{Surjectivity_Part_Details_2})
%holds. Proposition \ref{h'_theta}
%will be proved at the end of this Section.
%These will complete the proof of Theorem
%\ref{Surj_Etale_exc}.

In the constructions below we will not distinguish between a reduced closed subscheme of a reduced scheme
and corresponding closed subset of that reduced scheme.
\begin{constr}\label{Constr_alpha}
Let $Z_{1,red}$ be the closed subset $\{F=0\}$ in $W\times X'$ regarded
as a closed subset in $W\times_B \bb A^N_B=W\times \bb A^N$
(recall that we agreed above to write $W\times X'$ for $W\times_B X'$ and c.t.r).
Let $(\cc V,\rho: \cc V\to \bb A^N_B,s: X'\to \cc V)$ be the \'{e}tale neighborhood
of $X'$ in $\bb A^N_B$ as in Proposition
\ref{Framing_Of_X'}.
Then
$(W\times \cc V,id \times \rho: W\times \cc V\to W\times \bb A^N,id\times s: W\times X'\to W\times \cc V)$ %(edited) W\times \bb A^N_B changed to W\times \bb A^N
is an \'{e}tale neighborhood
of
$W\times X'$ in $W\times \bb A^N$. Hence %(edited) W\times \bb A^N_B changed to W\times \bb A^N
$(W\times \cc V,id \times \rho: W\times \cc V \to W\times \bb A^N,(id\times s)|_{Z_{1,red}}: Z_{1,red} \to W\times \cc V)$
is an \'{e}tale neighborhood
of
$Z_{1,red}$ in $W\times \bb A^N$.
For any $i=1,...,N-1$ we will write below for simplicity
$\phi_i$ for $pr^*_{\cc V}(\phi_i)$.
Set
$$\alpha=(Z_{1,red},id\times \rho:  W\times \cc V \to W\times \bb A^N,(id\times r)^*(F),\phi_1,...,\phi_{N-1};
W\times \cc V\xrightarrow{pr_{\cc V}} \cc V \xrightarrow{r} X') \in Fr_N(W,X'),$$
where $r: \cc V\to X'$ is as in Proposition
\ref{Framing_Of_X'}, $F\in k[W\times X']$ is as in
Proposition
\ref{h'_theta}.
By the item (f) of Proposition \ref{h'_theta} the morphism $\alpha$ is such that $\alpha|_{W-S}$ run inside of $X'-S'$.
%Thus $\Phi$ defines a morphism
%$\langle\langle \Phi \rangle\rangle \in \bb ZF_N((W,W-S)),(X',X'-S'))$
%Set
%$$\alpha=\langle\langle \Phi \rangle\rangle \in \bb ZF_N((W,W-S)),(X',X'-S')).$$
\end{constr}

\begin{constr}\label{Constr_beta_G}
Let $\{h_0=0\}$ be the closed subset in $W'\times X'$ as in Proposition
\ref{h'_theta}
regarded
as a closed subset in $W'\times_B \bb A^N_B=W'\times \bb A^N$.
By the item (b) of Proposition \ref{h'_theta} one has
$\{h_0=0\}=\Delta(W')\sqcup G_{red}$ and $G_{red}\subset W'\times (X'-S')$.
Set
$\cc V'=r^{-1}(X'-S')$, $r'=r|_{\cc V'}: \cc V'\to X'-S'$.
Hence
$$(W'\times \cc V'-\Delta(W'),id \times \rho: W'\times \cc V'-\Delta(W') \to W'\times \bb A^N,(id\times s)|_{G_{red}}: G_{red} \to
(W'\times \cc V'-\Delta(W'))$$ %(edited) W'\times \bb A^N_B changed to W'\times \bb A^N
is an \'{e}tale neighborhood of
$G_{red}$ in $W'\times \bb A^N$. Set $(W'\times \cc V')^{\circ}=W'\times \cc V'-\Delta(W')$ %(edited) W'\times \bb A^N_B changed to W'\times \bb A^N
and
$$\beta_G=(G_{red},id \times \rho: (W'\times \cc V')^{\circ} \to W'\times \bb A^N,(id\times r')^*(h_0|_{W'\times (X'-S')}),\underline {\phi^{\circ}}';
(W'\times \cc V')^{\circ}\xrightarrow{r'\circ pr_{\cc V'}}  X'-S'),$$ %(edited) W'\times \bb A^N_B changed to W'\times \bb A^N
where
$\underline {\phi^{\circ}}':=(\phi'_1,...,\phi'_{N-1})|_{(W'\times \cc V')^{\circ}}$
with $\phi'_i$'s as in Lemma \ref{e'}.
Clearly,
$\beta_G\in Fr_N(W',X'-S')$.
\end{constr}

\begin{prop}\label{Surjectivity_Part_Details_2_2}
The equality (\ref{Surjectivity_Part_Details_2}) is true for the morphisms
$\alpha$ and $\beta_G$ from the last two constructions.
\end{prop}

To prove this Proposition we need three more constructions and two lemmas stated below.
\begin{constr}\label{Constr_beta}
Let $Z'_{0,red}$ be the closed subset $\{h_0=0\}$ in $W'\times X'$ as in Proposition
\ref{h'_theta}
regarded
as a closed subset in $W'\times_B \bb A^N_B=W'\times \bb A^N$.
Then
$$(W'\times \cc V',id \times \rho: W'\times \cc V' \to W'\times \bb A^N,(id\times s)|_{Z'_{0,red}}: Z'_{0,red} \to
W'\times \cc V')$$ %(edited) W'\times \bb A^N_B changed to W'\times \bb A^N
is an \'{e}tale neighborhood
of
$Z'_{0,red}$ in $W'\times \bb A^N$. %(edited) W'\times \bb A^N_B changed to W'\times \bb A^N
Set
$$\beta:=(Z'_{0,red},id \times \rho: W'\times \cc V' \to W'\times \bb A^N,(id\times r)^*(h_0),\underline {\phi}';
W'\times \cc V'\xrightarrow{r\circ pr_{\cc V'}}  X' )\in Fr_N(W',X'),$$ %(edited) W'\times \bb A^N_B changed to W'\times \bb A^N
where $\underline {\phi}'=(\phi'_1,...,\phi'_{N-1})$
with $\phi'_i$'s as in Lemma \ref{e'}.
\end{constr}

\begin{constr}\label{Constr_H}
Let $Z'_{\theta}$ be the closed subset
$\{h_{\theta}=0\}$ in $\bb A^1\times W'\times X'$ as in Proposition
\ref{h'_theta}
regarded
as a closed subset in $\bb A^1\times W'\times_B \bb A^N_B=\bb A^1\times W'\times \bb A^N$.
Then
$$(\bb A^1\times W'\times \cc V,id \times \rho: \bb A^1\times W'\times \cc V \to \bb A^1\times W'\times \bb A^N,(id\times s)|_{Z'_{\theta,red}}: Z'_{\theta,red} \to
\bb A^1\times W'\times \cc V)$$ %(edited) W'\times \bb A^N_B changed to W'\times \bb A^N
is an \'{e}tale neighborhood
of
$Z'_{\theta,red}$ in $\bb A^1\times W'\times \bb A^N$. Set %(edited) W'\times \bb A^N_B changed to W'\times \bb A^N
$$H_{\theta}:=(Z'_{\theta,red},id \times \rho: \bb A^1\times W'\times \cc V \to \bb A^1\times W'\times \bb A^N,(id\times r)^*(h_\theta),\underline {\phi}'_{\theta};
\bb A^1\times W'\times \cc V\xrightarrow{r\circ pr_{\cc V}}  X' ),$$ %(edited) W'\times \bb A^N_B changed to W'\times \bb A^N
where $\underline {\phi}'_{\theta}=pr^*_{ W'\times \cc V}(\phi'_1,...,\phi'_{N-1})$
with $\phi'_i$'s as in Lemma \ref{e'}.
Clearly, $H_{\theta}\in Fr_N(\bb A^1\times W',X')$.
\end{constr}

\begin{constr}\label{Constr_gamma}
Set $(W'\times \cc V)^{\circ\circ}=W'\times \cc V-(id\times s)(G)$ and
$$\gamma=(\Delta(W'),id \times \rho:(W'\times \cc V)^{\circ\circ} \to W'\times \bb A^N,((id\times r)^*(h_0))|_{(W'\times \cc V)^{\circ\circ}}),\underline {\phi^{\circ\circ}}';
(W'\times \cc V)^{\circ\circ}\xrightarrow{r\circ pr_{\cc V}}  X'),$$ %(edited) W'\times \bb A^N_B changed to W'\times \bb A^N
where $\underline {\phi^{\circ\circ}}'=\underline {\phi}'|_{(W'\times \cc V)^{\circ\circ}}$ with $\underline {\phi}'$ as in Construction
\ref{Constr_beta}.
\end{constr}
The proofs of the following two lemmas are straightforward
\begin{lem}\label{alpha_pi_and_H1}
One has an equalities in $Fr_N(W',X')$
$$\alpha \circ \pi_{W}=(Z'_{1,red},id\times \rho:  W'\times \cc V \to W'\times \bb A^N,(id\times r)^*(h_1),\phi'_1,...,\phi'_{N-1};
W'\times \cc V\xrightarrow{pr_{\cc V}} \cc V \xrightarrow{r} X')=H_1,$$ %(edited) W'\times \bb A^N_B changed to W'\times \bb A^N
where $H_1:=H_{\theta}\circ i_1\in Fr_N(W',X')$.
%Moreover,
%$\alpha|_{W-S}$ run inside $X'-S'$ and $H_1|_{W'-S'}$ run inside $X'-S'$.
\end{lem}

\begin{lem}\label{H0_and_beta}
One has an equalities $H_0:=H_{\theta}\circ i_0=\beta$ in $Fr_N(W',X')$. Under the notation \ref{Main} one has an equality
$$\langle\beta \rangle=\langle j \rangle \circ \langle \beta_G \rangle + \langle \gamma \rangle \in \bb ZF_N(W',X').$$
%Moreover, under the notation from Definition \ref{} one has an equality in $\bb ZF_N((W',W'-S'),(X',X'-S'))$
%$$\langle\langle \beta \rangle\rangle= \langle\langle j \rangle\rangle \circ \langle\langle \beta_G \rangle\rangle + \langle\langle \gamma \rangle\rangle.$$
\end{lem}
%By Theorem \ref{Spriamlenie_1} one has equality
%$[[\gamma]]=[[can'']]\circ [[\sigma^N_{W'}]]\in \overline {\bb ZF}_N((W',W'-S'),(X',X'-S'))$.
%{\it Prove now the Claim.}
\begin{proof}(of Proposition \ref{Surjectivity_Part_Details_2_2}).
Using Proposition \ref{h'_theta}, Lemmas \ref{alpha_pi_and_H1}, \ref{H0_and_beta}, results, notation and definitions from Section
\ref{Notation_Agreements}
we get a chain of equalities in $\overline {\bb ZF}_N((W',W'-S'),(X',X'-S'))$
$$[[\alpha]]\circ [[\pi_{W}]]=[[H_1]]=[[H_0]]=[[\beta]]=[[j]]\circ [[\beta_G]]+[[\gamma]],$$
where $[[\beta_G]]\in \overline {\bb ZF}_N((W',W'-S'),(X'-S',X'-S'))$.
By the item (v) of Lemma \ref{e'} and
\cite[Theorem 13.3]{GP4} one has an equality
$[[\gamma]]=[[can'']]\circ [[\sigma^N_{W'}]]$
in $\overline {\bb ZF}_N((W',W'-S'),(X',X'-S'))$.
Proposition \ref{Surjectivity_Part_Details_2_2} is proved.
\end{proof}
\begin{proof}(of Theorem \ref{Surj_Etale_exc})
Let $inc: U\hookrightarrow W$ and $inc': U'\hookrightarrow W'$ be the inclusions.
Set
$a=\alpha\circ inc\in Fr_N(U,X')$, $b_G=\beta_G \circ inc' \in Fr_N(U',X'-S')$.
Clearly, $a|_{W-S}$ run inside of $X'-S'$ and $b_G|_{W'-S'}$ run inside of $X'-S'$
in the sense of Definition \ref{Runs_Inside}.
The equality (\ref{Surjectivity_Part_Details_2}) yields an equality
$$[[a]]\circ [[\pi_U]]-[[j]]\circ [[b_G]]=[[can']]\circ [[\sigma^N_U]] \in \overline {\bb ZF}_N((U',U'-S')),(X',X'-S')),$$
where $can': U'\hookrightarrow X'$, $\pi_U: U'\to U$ are as at the beginning of this Section.
{\it Thus the morphisms
$[[a]]$ and $[[b_G]]$ are such that the equality
(\ref{Surjectivity_Part_Details}) does hold.
}
This completes the proof of Theorem \ref{Surj_Etale_exc}.
\end{proof}

\section{Appendix}
Consider the elementary distinguished square
(\ref{Nisn_square}).
Let $S=X-V$ and $S'=X'-V'$ be closed subschemes equipped with reduced structures.
Let $x\in S$ and $x' \in S'$ be the closed points as in Section
\ref{section_surjectivity}
such that $\Pi(x')=x$.
%Let
%$U=Spec(\cc O_{X,x})$
%and
%$U'=Spec(\cc O_{X',x'})$.
{\it The main aim of this Section is to shrink $X$ and $X'$ appropriately and to 
construct a commutative diagram of the form (\ref{SquareDiagram_new}) 
subjecting to the conditions described right below the diagram 
(\ref{SquareDiagram_new}).
}

Let $in:X^{\circ}\hookrightarrow X$ and $in':
(X')^{\circ}\hookrightarrow X'$ be open such that
\begin{itemize}
\item[(1)]
$x\in X^{\circ}$,
\item[(2)]
$x'\in (X')^{\circ}$,
\item[(3)]
$\Pi((X')^{\circ})\subset X^{\circ}$,
\item[(4)]
the square
$$\xymatrix{V'\cap (X')^{\circ} \ar[r]\ar[d]& (X')^{\circ}\ar^{\Pi|_{(X')^{\circ}}}[d]\\
               V\cap X^{\circ}\ar[r]&X^{\circ}}$$
is an elementary distinguished square.
\end{itemize}

\begin{rem}
\label{Srinking_X_and_X'} One way of shrinking $X$ and $X'$ such
that properties $(1)-(4)$ are fulfilled is as follows. Replace $X$
by an affine open $X_{new}$ containing $x$ and then replace $X'$
by $X'_{new}=\Pi^{-1}(X_{new})$. 
For any $X$-scheme $p: T \to X$ we will write below $T_{new}$ for $p^{-1}(X_{new})$.
For any morphism of $r: T\to T'$ of $X$-schemes we will write $r_{new}$
for the obvious morphism $T_{new} \to T'_{new}$.
\end{rem}

Let $X'_n$ be the normalization of $X$ in $Spec(k(X'))$. Let $\Pi_n:
X'_n \to X$ be the corresponding finite morphism. Since $X'$ is
$k$-smooth it is an open subscheme of $X'_n$. Let $Y''=X'_n - X'$.
It is a closed subset in $X'_n$. Since $\Pi|_{S'}: S' \to S$ is an
isomorphism of schemes, then $S'$ is closed in $X'_n$. Thus $S' \cap Y''
= \emptyset$. Hence there is a function $f\in k[X'_n]$ such that
$f|_{Y''}=0$ and $f|_{S'}=1$.

\begin{defs}
\label{X'_0} Set $X'_{0}=(X'_n)_f$, $Y'=\{f=0\}$,
$Y=\Pi_n(Y'_{red})\subset X$. Note that $X'_{0}\subseteq X'$ and $X'_{0}$ is an affine
$k$-variety as a principal open subset of the affine $k$-variety
$X'_n$. We regard $Y'$ as an effective Cartier divisor of $X'_n$.
The subset $Y$ is closed in $X$, because $\Pi_n$ is finite. Set
$\Pi_{0}=\Pi|_{X'_{0}}$.
\end{defs}

\begin{rem}
\label{X'andX'_new} We have $\Pi^{-1}_{0}(S)=S'$. Therefore
the varieties $X$ and $X'_{0}$ are subject to properties $(1)-(4)$
of the present section. 
One has $X'_0=(X'_n)_f$, $X'_0=X'_n-Y'$, $Y=\Pi_n(Y'_{red})$ and $Y$ is closed in $X$.
Below we will work with this $X'_{0}$. 
However, we will write $X'$ for $X'_{0}$ and $\Pi$ for $\Pi_{0}$.

So, $X'=(X'_n)_f$, $X'=X'_n-Y'$, $Y=\Pi_n(Y'_{red})$ and $Y$ is closed in $X$.
\end{rem}

\begin{rem}
\label{Elementary_fibr} Take $X$ and $X'$ as at the end of the Remark
\ref{X'andX'_new}. Replacing $X$ with an affine open $X_{new}$ and $X'$ with $X'_{new}$ as described in Remark
\ref{Srinking_X_and_X'},
%in such a way that the condition $(1)$ from this section is secured
and using \cite[Prop.7.2]{GP4}
%\ref{CartesianDiagramArtin}
one can find an almost elementary fibration 
$q: X_{new}\to B$
in the sense of
\cite[Defn.7.1]{GP4}
%\ref{DefnElemFib}
%of~\cite{PSV}
(here $B$ is
affine open in $\bb P^{n-1}$). Moreover 
one can find an almost elementary fibration
$q: X_{new}\to B$
such that 
$$q|_{Y_{new}\cup S_{new}}: Y_{new}\cup S_{new} \to B$$ 
is finite,
$\omega_{B/k}\cong \cc O_B$, $\omega_{X_{new}/k}\cong \cc O_{X_{new}}$.

The scheme $X'_{new}$ will be regarded below as a $B$-scheme via
the morphism $q\circ \Pi_{new}$.
\end{rem}

%Let $X$ and $X'$ be as in Remark \ref{Elementary_fibr} and 
Let
$q: X_{new} \to B$ be the almost elementary fibration from
Remark \ref{Elementary_fibr}.
Since $q: X_{new} \to B$ is an almost elementary fibration
there is a commutative diagram of the form
(see \cite[Defn.7.1]{GP4})
\begin{equation}
\label{SquareDiagram_2}
    \xymatrix{
     X_{new}\ar[drr]_{q}\ar[rr]^{j}&&
\bar X_{new}\ar[d]_{\overline q}&&X_{new,\infty}\ar[ll]_{i}\ar[lld]^{q_{\infty}} &\\
     && B  &\\    }
\end{equation}
with morphisms $j$, $\overline q$, $i$, $q_{\infty}$ subjecting the conditions
(i)--(iv) from
\cite[Defn.7.1]{GP4}.

The composite morphism
$X'_{new} \xrightarrow{\Pi_{new}} X_{new} \xrightarrow{j} \bar X_{new}$
is quasi-finite. Let
$\bar X'_{new}$ be the normalization of $\bar X_{new}$ in
$Spec(k(X'_{new}))=Spec(k(X'))$.
Let $\bar \Pi: \bar X'_{new} \to \bar X_{new}$
be the canonical morphism (it is finite and surjective).
Then
$\bar \Pi^{-1}(X_{new})$
coincides with the normalization of $X_{new}$ in
$Spec(k(X'))$. 
So, 
$\bar \Pi^{-1}(X_{new})=X'_{n,new}$, 
$\bar \Pi|_{X'_{n,new}}=\Pi_{n,new}: X'_{n,new}\to X_{new}$
under the notation of the Remark 
\ref{Srinking_X_and_X'}
and $Y'_{new}\subseteq X'_{n,new}$ is a closed subscheme. 
Furthermore, 
$X'_{new}=X'_{n,new}-Y'_{new}$.

%Let $f':=f|_{(\bar \Pi)^{-1}(X)}$,
%where
%$f$ is from Definition
%\ref{X'_new}.
%Let $Y'_{new}=\{f'=0\}$ be the closed subscheme of $(\bar \Pi)^{-1}(X)$.
The morphism
$\bar q|_{Y'_{new}}: Y'_{new} \to B$ is finite,
since 
$Y'_{red,new}$ is finite over $Y_{new}$ and 
%$q|_{Y_{new}}: Y_{new} \to B$ is finite and
%$\Pi_{n,new}$
$Y_{new}$ 
is finite over $B$.
Thus $Y'_{new}$ is closed in $\bar X'_{new}$.
Since $Y'_{new}$ is in $(\bar \Pi)^{-1}(X)$
it has the empty intersection with $\bar \Pi^{-1}(X_{new,\infty})$.
Hence
$$X'_{new}= \bar X'_{new} - (\bar \Pi^{-1}(X_{new,\infty})\sqcup Y'_{new}).$$
Both $(\bar \Pi)^{-1}(X_{new,\infty})$ and $Y'_{new}$ are Cartier divisors in
$\bar X'_{new}$. The Cartier divisor $(\bar \Pi)^{-1}(X_{new,\infty})$ is
ample. Thus the Cartier divisor 
$X'_{new,\infty}:= (\bar \Pi)^{-1}(X_{new,\infty})\sqcup Y'_{new}$ is ample as well and 
$(\bar q\circ \bar \Pi)|_{X'_{new,\infty}}: X'_{new,\infty} \to B$ is finite. 
Set $\bar q'=\bar q\circ \bar \Pi$, $q'=\bar q'|_{X'_{new}}$, $q'_{\infty}=\bar q'|_{X'_{new,\infty}}$.
Let $j': X'_{new}\hookrightarrow \bar X'_{new}$
be the open embedding and 
$i': X'_{new,\infty}\hookrightarrow \bar X'_{new}$
be the closed embedding.
The following claim is obvious now.\\
{\it Claim.} The commutative diagram
\begin{equation}
\label{SquareDiagram_new_2}
    \xymatrix{
    X_{new} \ar[drrrr]_{q} && X'_{new}\ar[ll]_{\Pi_{new}} \ar[drr]^{q'}\ar[rr]^{j'}&& \overline X'_{new}\ar[d]_{\overline q'} && X'_{new,\infty}\ar[ll]_{i'}\ar[lld]_{q'_{\infty}} &\\
     &&&& B  &,\\   }
\end{equation}
subjects to the conditions described right below the diagram
(\ref{SquareDiagram_new}).

\end{document}